\documentclass[reqno, 12pt]{amsart}
\pdfoutput=1
\makeatletter
\let\origsection=\section \def\section{\@ifstar{\origsection*}{\mysection}} 
\def\mysection{\@startsection{section}{1}\z@{.7\linespacing\@plus\linespacing}{.5\linespacing}{\normalfont\scshape\centering\S}}
\makeatother        
\usepackage{amsmath,amssymb,amsthm,bm}    
\usepackage{mathabx}\changenotsign  
\usepackage{bbm, accents}  
\usepackage{mathrsfs} 
\usepackage{dsfont} 
\usepackage[babel]{microtype}
\usepackage{xcolor}  	
\usepackage[backref]{hyperref}
\hypersetup{
	colorlinks,
    linkcolor={red!60!black},
    citecolor={green!60!black},
    urlcolor={blue!60!black},
}

\usepackage{bookmark}

\usepackage[abbrev,msc-links,backrefs]{amsrefs} 
\usepackage{doi}

\renewcommand{\PrintDOI}[1]{\doi{#1}}

\usepackage[T1]{fontenc}
\usepackage{lmodern}

\usepackage[english]{babel}
\numberwithin{equation}{section}

\usepackage{geometry}
\usepackage{enumitem}

\let\polishlcross=\l
\def\l{\ifmmode\ell\else\polishlcross\fi}

\let\setminus=\smallsetminus

\makeatletter
\def\moverlay{\mathpalette\mov@rlay}
\def\mov@rlay#1#2{\leavevmode\vtop{   \baselineskip\z@skip \lineskiplimit-\maxdimen
   \ialign{\hfil$\m@th#1##$\hfil\cr#2\crcr}}}
\newcommand{\charfusion}[3][\mathord]{
    #1{\ifx#1\mathop\vphantom{#2}\fi
        \mathpalette\mov@rlay{#2\cr#3}
      }
    \ifx#1\mathop\expandafter\displaylimits\fi}
\makeatother

\newcommand{\dcup}{\charfusion[\mathbin]{\cup}{\cdot}}

\DeclareFontFamily{U}  {MnSymbolC}{}
\DeclareSymbolFont{MnSyC}         {U}  {MnSymbolC}{m}{n}
\DeclareFontShape{U}{MnSymbolC}{m}{n}{
    <-6>  MnSymbolC5
   <6-7>  MnSymbolC6
   <7-8>  MnSymbolC7
   <8-9>  MnSymbolC8
   <9-10> MnSymbolC9
  <10-12> MnSymbolC10
  <12->   MnSymbolC12}{}
\DeclareMathSymbol{\powerset}{\mathord}{MnSyC}{180}

\usepackage{tikz}
\usetikzlibrary{calc}
\pgfdeclarelayer{background}
\pgfdeclarelayer{foreground}
\pgfdeclarelayer{front}
\pgfsetlayers{background,main,foreground,front}
		
\theoremstyle{plain}
\newtheorem{thm}{Theorem}[section]
\newtheorem{fact}[thm]{Fact}
\newtheorem{prop}[thm]{Proposition}
\newtheorem{clm}[thm]{Claim}
\newtheorem{cor}[thm]{Corollary}
\newtheorem{lem}[thm]{Lemma}

\theoremstyle{definition}

\let\theta=\vartheta
\let\rho=\varrho
\let\phi=\varphi

\def\ZZ{\mathds Z} 

\def\RR{\mathds R}

\let\polishlcross=\l
\def\l{\ifmmode\ell\else\polishlcross\fi}

\def\lra{\longrightarrow}

\let\sm=\smallsetminus

\newtheoremstyle{note}  {4pt}  {4pt}  {\sl}  {}  {\itshape}  {.}  {.5em}          {}
\theoremstyle{note}

\begin{document}
\title[$K_4$-free graphs have sparse halves]{$K_4$-free graphs have sparse halves}

\author{Christian Reiher}
\address{Fachbereich Mathematik, Universit\"at Hamburg, Hamburg, Germany}
\email{Christian.Reiher@uni-hamburg.de}

\subjclass[2010]{}

\begin{abstract}
	Every $K_4$-free graph on $n$ vertices has a set of 
	$\lfloor n/2\rfloor$ vertices spanning at most~$n^2/18$ edges. 
\end{abstract}

\maketitle

\section{Introduction}\label{sec:introduction}

By a well-known theorem of Mantel~\cite{M} every graph on $n$ vertices with more 
than $\frac14n^2$ edges contains a triangle. As balanced complete bipartite graphs 
show, the number $\frac14 n^2$ appearing in this result is optimal. However, the 
only extremal graph for this problem contains a large independent 
set (of size $\lceil \frac12n\rceil$) and thus it is natural to wonder whether graphs 
with fewer yet more uniformly distributed edges need to contain triangles as well. 
An especially prominent and still open version of this question, due to Erd\H{o}s~\cite{E1}, 
asks whether every $n$-vertex graph~$G$ contains a triangle provided that all 
sets $X\subseteq V(G)$ of $|X|=\lfloor \frac12n\rfloor$ vertices span more 
than $\frac1{50}n^2$ edges. 
Balanced blow-ups of the pentagon show that the denominator $50$ would be optimal. 
Moreover, Simonovits observed that balanced blow-ups of the Petersen graph achieve 
equality as well (e.g.\ see~\cite{Er97a}). 
This so-called {\it sparse halves conjecture} attracted 
the attention of many researchers~\cites{BMPP2018, K, KeSu06, NY}; despite 
the recent investigations of Razborov~\cite{R} the problem remains elusive.

There is a similar situation for graphs not containing a $K_4$, i.e., four mutually adjacent 
vertices. Generalising Mantel's result Tur\'an~\cite{T} proved that every graph on $n$ 
vertices possessing more than $\frac13n^2$ edges contains a $K_4$. The extremal graphs for 
this problem are balanced tripartite graphs and thus they contain independent sets of size 
$\lceil \frac13n\rceil$ as well as sets of size $\lfloor \frac12n\rfloor$ spanning only 
about $\frac1{18}n^2$ edges (the precise value depending on the residue class of $n$ 
modulo $6$). This state of affairs prompted Chung and Graham~\cite{CG} and, independently, 
Erd\H{o}s et al.\ \cite{EFRS} to formulate a sparse halves conjecture for $K_4$-free graphs, 
which is resolved in this article and restated in the above abstract. 
A more precise version of our main result reads as follows.

\begin{thm}\label{thm:main}
	If a graph $G$ on $n$ vertices has the property that every set 
	$X\subseteq V(G)$ of size $|X|=\lfloor \frac12n\rfloor$ spans at 
	least $\frac1{18}n^2$ edges, then either $G$ contains a $K_4$ 
	or $n$ is divisible by $6$ and~$G$ is a tripartite Tur\'an graph.
\end{thm}

The most recent contribution to this problem is due to X.~Liu and J.~Ma~\cite{LM}, who 
proved Theorem~\ref{thm:main} under the additional assumption that $G$ is regular. 
Several parts of their argument do not depend on degree regularity and 
we shall utilise some of their results. Another important idea of our approach 
is to analyse the neighbourhoods of vertices by means of Ramsey-Tur\'an theory, 
appealing to a lemma due to \L uczak, Polcyn, and Reiher. 

Throughout the rest of this article we shall call an $n$-vertex graph $G$ {\it extremal} 
if it is $K_4$-free and every set $X\subseteq V(G)$ of size $|X|=\lfloor \frac12n\rfloor$ 
spans at least $\frac1{18}n^2$ edges. Thus we are to prove that every extremal graph is a 
tripartite Tur\'an graph whose number of vertices is divisible by~$6$. 

\section{Preliminaries}
\subsection{Parity} \label{subsec:par}
A well-known blow-up argument, which we recapitulate below, shows that it suffices to 
establish Theorem~\ref{thm:main} for even values of $n$. In other words, we only need 
to prove the following apparently weaker claim. 

\begin{prop}\label{p:main}
	Let $G$ be an extremal graph on $n$ vertives. If $n$ is even, then it is divisible by $3$ 
	and $G$ is a tripartite Tur\'an graph.
\end{prop}

The reason why this statement implies Theorem~\ref{thm:main} is as follows: If $n$ is odd
we construct a graph $H$ by replacing every vertex $x$ of $G$ by two new vertices $x'$, $x''$
and every edge $xy\in E(G)$ by all four possible edges from $\{x', x''\}$ to $\{y', y''\}$.
Evidently, $H$ is still $K_4$-free and it is not difficult to verify that any $n$ vertices 
of $H$ span at least $\frac29n^2$ edges. So Proposition~\ref{p:main} tells us that $2n$
is divisible by $3$ and that $H$ is a tripartite Tur\'an graph. Consequently, $G$ itself 
is the complete tripartite graph all of whose vertex classes have size $\frac 13n$. Recalling
that $n$ is odd we can form a set $X\subseteq V(G)$ of size $\lfloor \frac12n\rfloor$
by taking one of these vertex classes together with $\lfloor \frac16n\rfloor$ vertices from 
another vertex class. But now $e(X)=\frac 13n\cdot \lfloor \frac16n\rfloor < \frac 1{18}n^2$
contradicts the hypothesis. 

\subsection{Ramsey-Tur\'an theory}
Given natural numbers $n\ge s\ge 1$ and $r\ge 3$ Vera T.~S\'os asked to determine 
the largest number $f_r(n, s)$ of edges that a $K_r$-free graph on $n$ vertices 
can have provided that none of its independent sets consists of more than $s$ vertices. 
For instance, in the case $r=3$ one can exploit that in triangle-free graphs the 
neighbourhoods of all vertices are independent sets. This yields the {\it trivial 
bound} $f_3(n, s)\le \frac12ns$, which can be shown to be sharp in surprisingly many 
cases~\cite{B10}. 
The study of $K_4$-free graphs turned out to be much more difficult and following the 
works~\cites{BE76, EHSS, FLZ15, Sz72} 
it has been proved in~\cite{LR-a} 
that $f_4(n, \delta n)=\frac12\bigl(\frac14+\delta-\delta^2+o(1)\bigr)n^2$ holds for every sufficiently 
small $\delta>0$ and that there are analogous results for larger forbidden cliques.

When $\frac sn$ is large, however, even the case $r=3$ of triangle-free graphs is not 
completely understood. 
An old result due to Andr\'asfai~\cite{A} asserts $f_3(n, s)=n^2-4ns+5s^2$ 
whenever $s\in \bigl[\frac 25n, \frac 12n\bigr]$. This was recently reproved 
by \L uczak, Polcyn, and Reiher~\cite{LPR2}
and their alternative proof suggests the following result. 

%
%

\begin{lem}\label{lem:1243}
	Let $G=(V, E)$ be a triangle-free graph on $n$ vertices. If for some natural number 
	$m\le n$ it is the case that every set $X\subseteq V$ of size $m$ induces at 
	least $\frac29m^2$ edges, then there are disjoint independent sets $A, B\subseteq V$
	such that $t=|A|+|B|$ 
	satisfies 
	\begin{equation}\label{eq:1925}
		n\le \frac{2m}3+\frac{5t}8-\frac{2m^2}{9t}\le \frac m3+\frac{3t}4 \,.
	\end{equation}
	In particular, the assumption is only satisfiable for $m\ge \frac34 n$.
\end{lem}

Before beginning the proof we quote~\cite{LPR2}*{Lemma~2.2}. Here $\alpha(G)$ denotes 
the maximum size of an independent set in a graph $G$.

\begin{lem}[\L uczak, Polcyn, and Reiher]\label{l:isolating}
Given a graph $G$, suppose
	\begin{enumerate}
		\item[$\bullet$]  that $A\subseteq V(G)$ is an independent set of size $\alpha(G)$, 
		\item[$\bullet$] and that $M$ is a matching in $G$ from $V(G)\setminus A$ to $A$, 
			the size of which is as large as possible. 
	\end{enumerate}
	If $G'$ denotes the graph obtained from $G$ by isolating the vertices in 
	$A\setminus V(M)$, i.e., by deleting all edges incident with them, then 
	$\alpha (G')=\alpha(G)$. 
\end{lem}

\begin{proof}[Proof of Lemma~\ref{lem:1243}]
	Pick an independent set $A\subseteq V$ of size $|A|=\alpha(G)$ as well as an 
	independent set $B\subseteq V\sm A$ of size $|B|=\alpha(G-A)$. We shall prove 
	that $t=|A|+|B|$ has the desired property. 
	
	Before reaching this goal we need to establish several weaker estimates involving 
	the cardinalities of $A$ and $B$. First, the local density assumption yields 
	$|E|\ge \frac29 n^2$ and thus there exists a vertex~$x$ whose degree is at 
	least $\frac49 n$. As the neighbourhood of $x$ is independent, 
	we have $|A|\ge \tfrac 49 n>\frac25n$, whence $|A|>\tfrac23(n-|A|)$. 
	Since $e(A)=0<\frac29 m^2$ yields $|A|<m$, we can deduce 
	\begin{equation}\label{eq:1346} 
		(m-|A|)+\frac{2|A|(n-|A|)}{3(m-|A|)}
		>
		(m-|A|)+\frac{4(n-|A|)^2}{9(m-|A|)}
		\ge
		\frac 43(n-|A|)\,,
	\end{equation}
	where the last step utilises the AM-GM inequality. 
	
	Next, for every $W\subseteq V\sm A$ of 
	size $|W|=m-|A|$ we have $e(A\dcup W)\ge \frac29 m^2$, where the symbol ``$\dcup$''
	means ``disjoint union''. By averaging 
	over $W$ we obtain 
	\[
		\tfrac29 m^2
		\le 
		\frac{m-|A|}{n-|A|}e(A, V\sm A)+ \left(\frac{m-|A|}{n-|A|}\right)^2e(V\sm A)\,.
	\]
	Together with the estimates $e(A, V\sm A)\le |A||B|$ and $e(V\sm A)\le (n-|A|-|B|)|B|$,
	which follow from the fact that $|N(x)\sm A|\le \alpha(G-A)=|B|$ holds for every vertex 
	$x\in V$, we conclude 
	\begin{align*}
		\tfrac29 m^2
		&\le
		\frac{m-|A|}{n-|A|}|A||B|+\left(\frac{m-|A|}{n-|A|}\right)^2|B|(n-|A|-|B|) \\
		&=
		\frac{m-|A|}{n-|A|}|B|m-\left(\frac{m-|A|}{n-|A|}\right)^2|B|^2\,,
	\end{align*}
	which rewrites as
	\[
		0
		\le 
		\left(\frac{m-|A|}{n-|A|}|B|-\frac m3\right)
			\left(\frac{2m}3-\frac{m-|A|}{n-|A|}|B|\right)\,.
	\]
	This implies
	\begin{equation}\label{eq:1355}
		\frac{m-|A|}{n-|A|}|B|\ge \frac m3\,,
	\end{equation}
	whence 
	\[
		|A|+2|B|+m
		\ge
		|A|+\frac{2m}3\cdot \frac{n-|A|}{m-|A|}+m
		=
		\frac{2n}3+\frac{4|A|}3+(m-|A|)+\frac{2|A|(n-|A|)}{3(m-|A|)}\,.
	\]
	Due to~\eqref{eq:1346} this proves	
	\begin{equation}\label{eq:1350}
		|A|+2|B|+m 
		>
		2n\,.
	\end{equation}

	Finally, we remark that~\eqref{eq:1355} implies $|B|\ge \tfrac13 n$, for which reason 
	\begin{equation}\label{eq:1357}
		21t=21(|A|+|B|)\ge 21(\tfrac49+\tfrac13)n>16n\ge 12n+4m\,.
	\end{equation}

	After these preparations we are ready for the main argument. Let $M$ be a maximum 
	matching between $B$ and $V\sm (A\dcup B)$. Because of $|M|\le |V\sm (A\dcup B)|$ 
	we have 
	\[
		|B\sm V(M)|
		\ge 
		|B|-(n-|A|-|B|) 
		\overset{\eqref{eq:1350}}{>}
		n-m 
	\]
	and, consequently, there exist disjoint sets $X, Y\subseteq B\sm V(M)$
	such that $|X|=n-m$ and $|Y|=|A|+2|B|+m-2n$. Now $|V\sm X|=m$ entails
	\[
		\frac29 m^2
		\le 
		e(V\sm X)
		=
		\sum_{y\in Y}d(y)+\sum_{a\in A}|N(a)\sm (X\dcup Y)|+e(V-(A\dcup X\dcup Y))\,.
	\]
	These three summands can be estimated as follows. First, we trivially have 
	$d(y)\le \alpha(G)=|A|$ for every $y\in Y$. Second, by Lemma~\ref{l:isolating} 
	applied to $G-A$ here in place of $G$ there we know that the graph $H$ obtained
	from $G-A$ by deleting the edges incident with $X\dcup Y$ satisfies $\alpha(H)=|B|$.
	Now for every $a\in A$ the independence of $N(a)\cup X\cup Y$  in $H$ implies 
	$|N(a)\sm(X\dcup Y)|\le |B|-|X\dcup Y|=n-|A|-|B|$. Third, $|V-(A\dcup X\dcup Y)|=2(n-|A|-|B|)$
	and Mantel's theorem yield $e(V-(A\dcup X\dcup Y))\le (n-|A|-|B|)^2$. So altogether 
	we obtain
	\[
		\tfrac29 m^2
		\le 
		|Y||A|+|A|(n-|A|-|B|)+(n-|A|-|B|)^2\,.
	\]
	Owing to
	\begin{align*}
		4[|Y||A| &+|A|(n-|A|-|B|)]
		=
		4|A|(|B|+m-n) \\
		&=(|A|+|B|)(|A|+|B|+2m-2n)
		-(|A|-|B|)(|A|-|B|-2m+2n)
	\end{align*}
	and $|A|-|B|-2m+2n\ge |A|-|B|\ge 0$ this leads to 
	\[
		\tfrac29 m^2
		\le
		\tfrac14 (|A|+|B|)(|A|+|B|+2m-2n)+(n-|A|-|B|)^2\,.
	\]
	Now we multiply by $144$ and take $t=|A|+|B|$ into account, thus deriving
	\[
		32m^2
		\le 
		36t(t+2m-2n)+144(n-t)^2\,,
	\]
	i.e., 
	\begin{equation}\label{eq:1359}
		(3t-4m)^2
		\le 
		(21t-4m-12n)(9t+4m-12n)\,.
	\end{equation}
	By~\eqref{eq:1357} the first factor on the right side is positive and thus 
	we obtain $9t+4m-12n\ge 0$. Combined with 
	\[
		(21t-4m-12n)+(9t+4m-12n)=30t-24n\le 6t
	\]
	this implies $21t-4m-12n\le 6t$ and substituting this back into~\eqref{eq:1359}
	we learn 
	\[
		(3t-4m)^2\le 6t(9t+4m-12n)\,, 
	\]
	i.e., 
	\[
		n
		\le 
		\tfrac34t+\tfrac13m-\tfrac 1{72t}(3t-4m)^2
		\le
		\tfrac34t+\tfrac13m\,.
	\]

	As the term in the middle simplifies to $\frac23 m+\frac 58 t-\frac{2m^2}{9t}$
	this completes the proof of~\eqref{eq:1925}. Moreover $n\le \frac34n+\frac13m$
	implies $m\ge \frac34n$. 
\end{proof}

\subsection{Edges in neighbourhoods}

In this subsection we study lower bounds on the number of edges spanned by the 
neighbourhood of a vertex in an extremal graph. We commence with the following
variant of a result by X.~Liu and J.~Ma, see~\cite{LM}*{Theorem~1.5(1)}.

\begin{lem}\label{lem:2317}
	Let $m$, $n$, $q$ be positive integers such that $q\ge \frac29m^2$ and $n\ge m$. 
	If a triangle-free graph $G$ has the property that every set $X\subseteq V(G)$
	of size $|X|=m$ spans at least $q$ edges, then $e(G)\ge \frac{nq}{2m-n}$. 
\end{lem}

\begin{proof}
	Lemma~\ref{lem:1243} informs us that $m\ge \frac34n$ and for this reason 
	the division by $2m-n$ is permissible. Now, arguing indirectly, we consider 
	a counterexample $(n, m, q, G)$ such that $n$
	is minimal. 
		
	\smallskip
	
	{\it \hskip 1em First Case: There is a vertex $x$ whose degree is less than $n-m$.}
	
	\smallskip
	
	Set $d=d(x)$ and $G'=G-x$. Because of $9d\le 9(n-m-1)\le 12m-9m-9<4m-2$
	we have 
	\[
		q-d
		> 
		\frac{2m^2}9-\frac{4m-2}9
		=
		\frac{2(m-1)^2}9\,.
	\]
	Moreover, every set $X\subseteq V(G')$ of size $|X|=m-1$ satisfies
	$e(X)\ge e(X\dcup\{x\})-d\ge q-d$. For these reasons the quadruple 
	$(n-1, m-1, q-d, G')$ satisfies the assumptions and the minimality of $n$ 
	discloses 
	\[	
		e(G')\ge \frac{(n-1)(q-d)}{2m-n-1}\,.
	\]
	Together with
	\[
		q
		\ge 
		\frac{2m^2-(4m-3n)(5m-3n)}9
		=
		(2m-n)(n-m)
		>
		(2m-n)d
	\]
	this shows
	\begin{align*}
		e(G)
		&=
		e(G')+d
		\ge
		\frac{(n-1)(q-d)}{2m-n-1}+d
		=
		\frac{(n-1)q}{2m-n-1}-\frac{2(n-m)d}{2m-n-1} \\
		&>
		\frac{(n-1)q}{2m-n-1}-\frac{2(n-m)q}{(2m-n-1)(2m-n)}
		=
		\frac{nq}{2m-n}\,,
	\end{align*}
 	contrary to $(n, m, q, G)$ being a counterexample. 
	
	\smallskip
	
	{\it \hskip 1em Second Case: The minimum degree of $G$ is at least $n-m$.}
	
	\smallskip
	
	Consider a vertex $x\in V(G)$. For every set $R\subseteq N(x)$ of size $m+d(x)-n$
	the set $R\dcup (V\sm N(x))$ has size $m$ and thus it spans at least $q$ edges. 
	By averaging over $R$ we infer 
	\[
		q\le e(V\sm N(x))+\frac{m+d(x)-n}{d(x)}e(N(x), V\sm N(x))\,,
	\]
	whence 
	\[
		\frac{n-m}{d(x)}e(N(x), V\sm N(x))\le e(G)-q\,.
	\]
	Summing over $x$ we obtain
	\[
		(n-m)\sum_{x\in V(G)}\frac 1{d(x)}\sum_{y\in N(x)}d(y)
		\le 
		n(e(G)-q)\,.
	\]
	As the left side rewrites as 
	\[
		(n-m)\sum_{xy\in E(G)}\left(\frac{d(y)}{d(x)}+\frac{d(x)}{d(y)}\right)
	\]
	we are led to $2(n-m)e(G)\le n(e(G)-q)$, which implies the desired lower bound
	on $e(G)$. 
\end{proof}

\begin{cor}\label{c:22}
	Let $G$ be an extremal graph on an even number $n$ of vertices. 
	If $x\in V$ has degree $d(x)\ge n/2$, then 
	\[
		e(N(x))\ge \frac{n^2}{18}\cdot \frac{d(x)}{n-d(x)}
	\]
	and there exist two disjoint independent sets $A, B\subseteq N(x)$ such 
	that $|A|+|B|\ge \frac 43d(x)-\frac 29n$.

	Moreover, if $Y\subseteq V(G)$ of size $|Y|\ge \frac12n$ induces no 
	triangles, then
	\[
		e(Y)\ge \frac{n^2}{18}\cdot \frac{|Y|}{n-|Y|}\,.
	\]
\end{cor}

\begin{proof}
	The subgraph $G'$ of $G$ induced by $N(x)$ is triangle-free.
	Now the statements addressing $N(x)$ follow from 
	Lemma~\ref{lem:2317} applied to $n/2$, $d(x)$, $n^2/18$, and $G'$ here in 
	place of $m$, $n$, $q$, and $G$ there and from Lemma~\ref{lem:1243}. 
	Similarly, Lemma~\ref{lem:2317} implies the last claim.
\end{proof}

We proceed with a lower bound on the number of edges spanned by ``small'' sets. 

\begin{lem}\label{lem:2248}
	Let $G$ be an extremal graph on an even number $n$ of vertices. 
	If the size of $X\subseteq V(G)$ belongs to $\bigl[\frac13n, \frac12n\bigr]$, 
	then $e(X)\ge \frac1{18}(3|X|-n)(6|X|-n)$.
\end{lem}

\begin{proof}
	Suppose first that there is a set $A\subseteq V(G)\sm X$ of size $|A|=\frac12n-|X|$ 
	such that $|N(a)\cap X|\le \frac23|X|$ holds for every $a\in A$. 
	Because of $|A\dcup X|=\frac12n$ we have 
	\[
		\tfrac1{18}n^2
		\le 
		e(A\dcup X)\le e(A)+\tfrac23|A||X|+e(X)\,.
	\]
	Since Tur\'an's theorem and $|A|\le |X|$ imply $e(A)\le \frac13|A|^2\le \frac13|A||X|$,
	this shows indeed that 
	\[
		e(X)
		\ge
		\tfrac1{18}n^2-|A||X|
		=
		\tfrac1{18}(3|X|-n)(6|X|-n)\,.
	\]

	If no such set $A$ exists, then there is a set $B\subseteq V(G)\sm X$
	of size $|B|=\frac12n$ such that every $b\in B$ satisfies $|N(b)\cap X| > \frac23|X|$. 
	If some three vertices $b',b'',b'''\in B$ form a triangle, then 
	\[
		|N(b')\cap X|+|N(b'')\cap X|+|N(b''')\cap X|> 2|X|
	\]
	implies that for some $x\in X$ we have a clique $b'b''b'''x$ of order four in $G$, 
	which is absurd. So~$B$ induces a triangle-free subgraph of $G$. Owing to $e(B)\ge n^2/18$
	there exists some vertex~$b_\star\in B$ 
	such that $|N(b_\star)\cap B|\ge \frac29n>\frac12n-|X|$.
	Pick an arbitrary set $W\subseteq N(b_\star)\cap B$ of size $|W|=\frac12n-|X|$. 
	Together with $b_\star$ any edge connecting two vertices of $W$ would yield
	a triangle whose three vertices are in $B$. So $W$ is independent and 
	\[
		\frac1{18}n^2
		\le 
		e(X\dcup W)
		=
		e(X)+e(W, X)
		\le 
		e(X)+|W||X|\,,
	\]
	as desired. 
\end{proof}

\begin{cor}\label{c:26}
	If $G$ denotes an extremal graph on an even number $n$ of vertices,
	then $\alpha(G)\le \frac13n$.
\end{cor}

\begin{proof}
	If $X\subseteq V(G)$ satisfies $|X|\in \bigl(\frac13n, \frac12n\bigr]$,
	then Lemma~\ref{lem:2248} yields $e(X)>0$, so $X$ cannot be independent.
\end{proof}

\subsection{Inequalities} In this subsection we prepare a stability analysis
of extremal graphs $G$ satisfying $e(G)\ge \frac7{24} v(G)^2$. This will involve 
the following inequality.  

\begin{fact}\label{fact:2250}
	Define $h\colon \bigl[\frac13, \frac12\bigr]\lra \RR$ by $h(\xi)=(3\xi-1)(6\xi-1)$.
	If $\xi\in \bigl[\frac13, \frac12\bigr]$ and $\gamma\in \bigl[\frac7{24}, \frac13\bigr]$, 
	then 
	\begin{equation}\label{eq:23}
		\frac{h(\xi)^2}\xi
		\ge 
		\frac{\xi(1+2\gamma)-8\gamma^2}{(1-2\gamma)^3}\,.
	\end{equation}
\end{fact}

\begin{proof}
	Writing $F(\gamma, \xi)$ for the right side of~\eqref{eq:23} we have 
	\[
		\frac{\partial F(\gamma, \xi)}{\partial \gamma}
		=
		-\frac{8(1+\gamma)(2\gamma-\xi)}{(1-2\gamma)^4}
		<
		0
	\]
	and, in particular, $F(\gamma, \xi)$ is decreasing in $\gamma$. 
	As the left side of~\eqref{eq:23} is nonnegative, this shows that  
	it suffices to prove 
	\begin{equation}
		\frac{h(\xi)^2}\xi
		\ge
		\frac{25}{24} F\bigl(\tfrac 7{24}, \xi\bigr)
		=
		\frac{114\xi-49}{5}\,.
	\end{equation}  
   Due to
   \[
   	100(h(\xi)-\xi)
		=
		2(30\xi-13)^2+(7-10\xi)+5(114\xi-49)
	\]
	we have indeed 
   \[
   	\frac{h(\xi)^2}\xi
		\ge
		4(h(\xi)-\xi)
		>
		\frac{114\xi-49}{5}\,. \qedhere
	\]
\end{proof}

\begin{lem}\label{l:2525}
	Let $G$ be an extremal graph on an even number~$n$ of vertices which 
	has $\gamma n^2$ edges. If $\gamma\ge \frac7{24}$, then 
	\begin{equation}\label{eq:25}
		e(N(x))^2
		\ge
		\frac{d(x) n^2}{324}\cdot \frac{(1+2\gamma)d(x)-8\gamma^2n}{(1-2\gamma)^3}
	\end{equation}
	holds for every vertex $x\in V(G)$.
\end{lem}

\begin{proof}
	If $d(x)<\frac13n$, then because of $1+2\gamma-24\gamma^2=(1-4\gamma)(1+6\gamma)<0$
	the right side of~\eqref{eq:25} cannot be positive and the claim is clear. 
	
	Suppose next that $d(x)=\xi n$ holds for some $\xi\in \bigl[\frac13, \frac12\bigr]$.
	Recall that Tur\'an's theorem implies $\gamma\le \frac13$. Lemma~\ref{lem:2248} yields 
	$e(N(x))\ge \frac1{18}h(\xi)n^2\ge 0$, where $h$ indicates the function studied in 
	Fact~\ref{fact:2250}, and thus we have indeed 
	\[
		\frac{e(N(x))^2}{d(x)}
		\ge
		\frac{n^3}{324}\cdot \frac{h(\xi)^2}{\xi}
		\ge
		\frac{n^2}{324}\cdot \frac{(1+2\gamma)d(x)-8\gamma^2n}{(1-2\gamma)^3}\,.  
	\]

	Finally, if $d(x)\ge \frac12n$, then Corollary~\ref{c:22} yields  
	\[
		e(N(x))^2
		\ge
		\frac{d(x) n^2}{324}\cdot\frac{d(x)n^2}{\bigl(n-d(x)\bigr)^2}
	\]
	and it remains to observe that
	\[
		\frac{d(x) n^2}{\bigl(n-d(x)\bigr)^2}
		-
		\frac{(1+2\gamma)d(x)-8\gamma^2n}{(1-2\gamma)^3}
		=
		\frac{\bigl(d(x)-2\gamma n\bigr)^2\bigl(2n-(1+2\gamma)d(x)\bigr)}
			{\bigl(n-d(x)\bigr)^2(1-2\gamma)^3}
	\]
	is nonnegative because of $\gamma\le\frac12$ and $d(x)\le n$. 
\end{proof}

\section{Two independent sets}

In this section we explore some consequences of the following observation. 

\begin{lem}\label{p:31}
	Suppose that $G=(V, E)$ denotes an extremal graph on an even number~$n$ 
	of vertices. If $A, B\subseteq V$ are two disjoint independent sets, then 
	$e(A, B)\le |E|-\frac29 n^2$.
\end{lem}

\begin{proof}
	Set $C=V\sm (A\dcup B)$ and notice that $|B| < \frac12n$ implies
	$|C| > \frac12n-|A|$. Now every set $Z\subseteq C$ of size $|Z|=\frac12n-|A|$
	satisfies $e(A\dcup Z)\ge \frac1{18}n^2$. Averaging over $Z$ we learn
	\[
		\frac{n^2}{18}
		\le 
		\frac{\frac12n-|A|}{|C|}e(A, C)+\left(\frac{\frac12n-|A|}{|C|}\right)^2e(C)\,,
	\]
	whence 
	\[
		\frac{n^2|C|}{9(n-2|A|)}
		\le 
		e(A, C)+\frac{\frac12n-|A|}{|C|}e(C)\,.
	\]
	This remains valid if we replace $A$ by $B$ and by adding both 
	estimates we obtain 
	\begin{align*}
		\frac{n^2|C|}{9}\left(\frac 1{n-2|A|}+\frac{1}{n-2|B|}\right)
		& \le 
		e(A, C)+e(B, C)+\frac{n-|A|-|B|}{|C|}e(C) \\
		&=
		|E|-e(A, B)\,.
	\end{align*}
	Together with 
	\[
		\frac 1{n-2|A|}+\frac 1{n-2|B|}
		\ge 
		\frac 4{(n-2|A|)+(n-2|B|)}
		=
		\frac 2{|C|}
	\]
	this proves the desired inequality.
\end{proof}

X.~Liu and J.~Ma proved in~\cite{LM}*{Theorem~4.1} that $e(G) > \frac14v(G)^2$ 
holds for every extremal graph $G$. In fact, they even obtained such a result 
with $\frac14$ replaced by $\frac{32}{123}$. Utilising Lemma~\ref{p:31} 
one can go slightly further. 

\begin{lem}\label{l:33}
Every extremal graph $G$ on an even number $n$ of vertices 
has at least $\frac7{24}n^2$ edges. 
\end{lem}

\begin{proof}
	Define $\gamma>0$ by $|E(G)|=\gamma n^2$ and recall that we already know $\gamma > \frac14$.  
	By averaging there exists a vertex $x\in V(G)$ satisfying $d(x)\ge 2\gamma n > \frac12n$.
	Corollary~\ref{c:22} leads to two disjoint independent sets 
	$A, B\subseteq N(x)$ such that the quantity $t=|A|+|B|$ satisfies 
	\begin{equation}\label{eq:33}
		t
		\ge 
		\tfrac43 d(x)-\tfrac 29 n
		\ge 
		(\tfrac 83\gamma-\tfrac 29)n
		>
		\tfrac49n\,.
	\end{equation}

	\smallskip
	
	{\it \hskip 1em First Case: $t<\frac12n$}
	
	\smallskip
	
	For $X=A\dcup B$ Lemma~\ref{lem:2248} yields
	\[
		36e(A, B) 
		\ge
		2(n^2-9nt+18t^2)
		=
		(3n-6t)^2+(18t-7n)n
		\overset{\eqref{eq:33}}{\ge}
		(48\gamma-11)n^2\,.
	\]
	Together with Lemma~\ref{p:31} this tells us $48\gamma-11\le 36\gamma-8$, i.e., 
	$\gamma\le \frac14$, which is absurd.
	
	\smallskip
	
	{\it \hskip 1em Second Case: $t\ge \frac12n$}
	
	\smallskip
	
	By Corollary~\ref{c:22} applied to $Y=A\dcup B$ and by Lemma~\ref{p:31}
	we have 
	\[
		\frac{t}{n-t}
		\le 
		\frac{18 e(A, B)}{n^2}    
		\le 
		18\gamma-4\,,
	\]
	whence 
	\[
		9n
		\le 
		9(n-t)(18\gamma-3)
		\overset{\eqref{eq:33}}{\le} 
		(11-24\gamma)(18\gamma-3)n\,,
	\]
	i.e.,
	\[
		0
		\le 
		(11-24\gamma)(6\gamma-1)-3
		=
		2(1-3\gamma)(24\gamma-7)\,,
	\]
	which proves that $\gamma\in\bigl[\frac7{24}, \frac13\bigr]$.
\end{proof}

We conclude this section with a closely related result. 

\begin{lem}\label{l:37}
	Let $G$ be an extremal graph on an even number $n$ of vertices with $\gamma n^2$ edges. 
	If $A, B\subseteq V(G)$ are two disjoint independent sets, then $|A|+|B|\le 2\gamma n$. 	
	Moreover, the maximum degree of $G$ is at most $(\frac32\gamma+\frac16)n$. 
\end{lem}

\begin{proof}
	Assume for the sake of contradiction that $t=|A|+|B|$ satisfies $t>2\gamma n\ge \frac12n$.
	As in the second case of the previous proof one obtains 
	$n\le (n-t)(18\gamma-3)<(1-2\gamma)(18\gamma-3)n$, which yields the 
	contradiction $(6\gamma-2)^2=1-(1-2\gamma)(18\gamma-3)<0$.
	
	Now for every vertex $x\in V(G)$ with $d(x)\ge \frac 12n$ Corollary~\ref{c:22} 
	yields $2\gamma n\ge \frac43 d(x)-\frac29n$, whence $d(x)\le (\frac32\gamma+\frac16)n$. 
	Due the $\frac32\gamma+\frac16>\frac 12$ the claim on the maximum degree of $G$ follows.
\end{proof}
\section{Three independent sets}

We proceed with a stability result addressing extremal graphs. 
With some additional assumptions on the degree distribution this 
was obtained earlier by X.~Liu and J.~Ma in~\cite{LM}*{Lemma~4.10}.

\begin{lem}\label{l:41}
	Suppose that $G=(V, E)$ is an extremal graph on an even number $n$ of vertices. 
	If $|E|=\gamma n^2$, then there 
	are three mutually disjoint independent sets $V_1, V_2, V_3\subseteq V(G)$
	such that 
	\[
		|V_1|+|V_2|+|V_3|
		\ge
		\frac{n}{3(1-2\gamma)}\,.
	\] 
\end{lem}

\begin{proof}
	Recall that Lemma~\ref{l:33} yields $\gamma\ge\frac 7{24}$.
	For every edge $xy\in E$ we denote the number of triangles containing it by $t_{xy}$.
	Similarly, for every vertex $x$ we write $t(x)$ for the number of triangles containing $x$
	or, in other words, for the number of edges induced by $N(x)$.
	
	For every non-isolated vertex $x\in V$ the Cauchy-Schwarz inequality implies 
	\[
		\sum_{y\in N(x)}t_{xy}^2
		\ge
		\frac{4t(x)^2}{d(x)}
	\]
	and in view of Lemma~\ref{l:2525} we obtain
	\begin{equation}\label{eq:3242}
		\sum_{y\in N(x)}t_{xy}^2
		\ge 
		\frac{n^2}{81}\cdot \frac{(1+2\gamma)d(x)-8\gamma^2n}{(1-2\gamma)^3}\,.
	\end{equation}
	This estimate clearly holds for isolated vertices $x\in V$ as well. 
	Summing~\eqref{eq:3242} over all vertices~$x$ we learn
	\[
		\sum_{xy\in E}t_{xy}^2
		\ge 
		\frac{n^2}{162} \cdot \frac{2(1+2\gamma)\gamma n^2-8\gamma^2n^2}{(1-2\gamma)^3}
		=
		\frac{\gamma n^4}{81(1-2\gamma)^2}\,.
	\]
	Multiplying by 
	\[
		\gamma n^2\cdot \sum_{xy\in E}t_{xy}^2
		\ge
		\left(\sum_{xy\in E}t_{xy}\right)^2
	\]
	and simplifying we deduce 
	\[
		\sum_{xy\in E}t_{xy}^2 
		\ge 
		\frac n{9(1-2\gamma)} \cdot \sum_{xy\in E}t_{xy}\,.
	\]

	Let $T$ denote the set of triangles in $G$. Due to $3|T|=\sum_{xy\in E}t_{xy}$
	the preceding estimate discloses 
	\[
		\sum_{xyz\in T}\bigl(t_{xy}+t_{yz}+t_{zx}\bigr)
		\ge
		\frac {n|T|}{3(1-2\gamma)}
	\]
	and thus there exists some triangle $xyz\in T$ such that 
	\[	
		t_{xy}+t_{yz}+t_{zx}
		\ge 
		\frac{n}{3(1-2\gamma)}\,.
	\]
	Now $V_1=N(y)\cap N(z)$, 
	$V_2=N(z)\cap N(x)$, and $V_3=N(x)\cap N(y)$ are the desired independent sets. 
\end{proof}

The next result occurs implicitly in the work of X.~Liu and J.~Ma,
or more precisely in their discussion of Case~1 in the proof of~\cite{LM}*{Theorem~4.9}. 
For the sake of completeness we include the short argument. 

\begin{lem}\label{l:42}
	Let $G=(V, E)$ be an extremal graph on an even number $n$ of vertices. 
	If a partition $V=V_1\dcup V_2\dcup V_3\dcup Z$ has the properties that 
	$V_i$ is independent and  $|V_i|+|V_{i+1}|\ge \frac12n$ 
	holds for every $i\in \ZZ/3\ZZ$, then  
	\[
		|E|\ge \frac{n(n^2+n|Z|+|Z|^2)}{3(n+2|Z|)}\,.
	\]
\end{lem}

\begin{proof}
	Define the real numbers $\gamma$ and $z$ such that $|E|=\gamma n^2$ and $|Z|=zn$. 
	Owing to $|V_1|+|V_2|+|V_3|\ge \frac34n$ we have $z\le \frac14$.
	For every $i\in \ZZ/3\ZZ$ we can apply Corollary~\ref{c:22} to 
	the set $Y=V_i\dcup V_{i+1}$, thus inferring 
	\[
		e(V_i, V_{i+1})\ge \frac{n^2}{18}\cdot \frac{|V_i|+|V_{i+1}|}{n-|V_i|-|V_{i+1}|}\,.
	\]
	Due to the Cauchy-Schwarz inequality, this implies 
	\begin{align*}
		& \phantom{\ge}  (n+2|Z|) \sum_{i\in\ZZ/3\ZZ} (|V_i|+|V_{i+1}|) \,e(V_i, V_{i+1}) \\
		&\ge 
		\frac{n^2}{18} \sum_{i\in\ZZ/3\ZZ} (n-|V_i|-|V_{i+1}|) \times 
		\sum_{i\in\ZZ/3\ZZ} \frac{(|V_i|+|V_{i+1}|)^2}{n-|V_i|-|V_{i+1}|} \\
		&\ge
		\frac{n^2}{18} \Bigl(\sum_{i\in\ZZ/3\ZZ} (|V_i|+|V_{i+1}|)\Bigr)^2
		=
		\frac{2n^2}{9}(n-|Z|)^2\,,
	\end{align*}
	whence
	\begin{equation}\label{eq:41}
		\sum_{i\in\ZZ/3\ZZ} (|V_i|+|V_{i+1}|)\, e(V_i, V_{i+1})
		\ge
		\frac{2n^3(1-z)^2}{9(1+2z)}\,.
	\end{equation}
		
	Let us now consider any two distinct indices $i, j\in \ZZ/3\ZZ$. 
	In view of ${|V_{i-1}|+|V_{i+1}|\ge \frac12n}$ we 
	have $|V_i|+|Z|\le \frac12n$. Moreover, 
	every set $Q\subseteq V_j$ of size $|Q|=\tfrac 12n-|V_i|-|Z|$ 
	satisfies 
	\[
		\frac{n^2}{18}
		\le 
		e(V_i\dcup Q\dcup Z)
		=
		e(V_i, Z)+e(Z)+e(V_i, Q)+e(Q, Z)\,.
	\]
	Multiplying by $|V_j|$ and averaging over $Q$ we learn 
	\[
		\frac{n^2|V_j|}{18}
		\le 
		|V_j|\bigl(e(V_i, Z)+e(Z)\bigr)
		+(\tfrac12 n-|V_i|-|Z|)\bigl(e(V_i, V_j)+e(V_j, Z)\bigr)\,.
	\]
	As $i$ and $j$ vary there arise six estimates of this form whose sum simplifies to
	\[
		\frac{n^3(1-z)}{9}
		\le
		\sum_{i\in\ZZ/3\ZZ}(|V_{i-1}|-|Z|)e(V_i, V_{i+1})
		+
		(n-2|Z|)\sum_{i\in\ZZ/3\ZZ}e(V_i, Z)
		+2(n-|Z|)e(Z)\,.
	\]
	Adding~\eqref{eq:41} and taking into account that Tur\'an's theorem yields 
	$e(Z)\le \frac13|Z|^2$ we obtain
	\[
		\frac{n^3(1-z)}{9}+\frac{2n^3(1-z)^2}{9(1+2z)}
		\le
		\frac{n^3z^2}{3}+(1-2z)\gamma n^3\,.
	\]
	Since $z\le\frac14<\frac12$, this is equivalent to 
	\[
		\gamma\ge \frac{1+z+z^2}{3(1+2z)}\,. \qedhere
	\]
\end{proof}

Our last preparatory result analyses partitions of extremal graphs into 
three ``almost independent'' sets. 

\begin{lem}\label{l:43}
	Let $G=(V, E)$ be an extremal graph on an even number $n$ of vertices. 
	If a partition $V=A_1\dcup A_2\dcup A_3$ satisfies 
	$e(A_1)+e(A_2)+e(A_3)=\omega n^2$ for some $\omega\le \frac1{60}$,
	then $|E|\ge (\frac13-\frac{29}{18}\omega)n^2$.
\end{lem}

\begin{proof}
	As usual we define $\gamma$ such that $|E|=\gamma n^2$.
	Notice that $e(A_i)\le\frac1{60}n^2$ yields $|A_i|<\tfrac12n$ 
	for every $i\in \ZZ/3\ZZ$. 
	If $i, j\in \ZZ/3\ZZ$ are distinct and $T\subseteq A_j$ has size 
	$|T|=\frac12n-|A_i|$, then
	\[
		\frac{n^2}{18}\le e(A_i\dcup T)=e(A_i)+e(A_i, T)+e(T)\,.
	\]
	By averaging over $T$ and multiplying with $|A_j|/(\frac12n-|A_i|)$ we deduce
	\[
		 \frac{|A_j|(n^2-18e(A_i))}{9(n-2|A_i|)}
		 \le
		 e(A_i, A_j)+\frac{\frac12n-|A_i|}{|A_j|}e(A_j)\,.
	\]
	Again there are six estimates of this form and this time their addition yields
	\[
		\sum_{i\in\ZZ/3\ZZ}\frac{(n-|A_i|)(n^2-18e(A_i))}{9(n-2|A_i|)}
		\le
		\sum_{i\in\ZZ/3\ZZ}\bigl(2e(A_i, A_{i+1})+e(A_i)\bigr)\,.
	\]
	Because of the identical equation 
	\[
		\frac{(n-|A_i|)(n^2-18e(A_i))}{9(n-2|A_i|)}+e(A_i)
		=
		\frac{n^2}{18}+\frac{n(n^2-18e(A_i))}{18(n-2|A_i|)}
	\]
	this implies 
	\[
		\sum_{i\in\ZZ/3\ZZ}\frac{n^2-18e(A_i)}{n-2|A_i|}
		\le
		(36\gamma-3)n\,.
	\]
	Owing to the Cauchy-Schwarz inequality we have 
	\[
		\Bigl(\sum_{i\in\ZZ/3\ZZ}\sqrt{n^2-18e(A_i)}\Bigr)^2
		\le
		\sum_{i\in\ZZ/3\ZZ}\frac{n^2-18e(A_i)}{n-2|A_i|}\times 
		\sum_{i\in\ZZ/3\ZZ}(n-2|A_i|)
	\]
	and the concavity of the square root entails (by Karamata's inequality)
	\[
		(2+\sqrt{1-18\omega})n\le \sum_{i\in\ZZ/3\ZZ}\sqrt{n^2-18e(A_i)}\,.
	\]
	Altogether we have thereby proved
	\[	
		(2+\sqrt{1-18\omega})^2\le 36\gamma-3\,,
	\]
	i.e., $8-18\omega+4\sqrt{1-18\omega}\le 36\gamma$. Since $\sqrt{1-x}\ge 1-\frac59x$
	holds for all $x\in \bigl[0, \frac3{10}\bigr]$, this implies
	%
	\[
		36\gamma\ge 8-18\omega+4(1-10\omega)=12-58\omega\,,
	\]
	as desired. 
\end{proof}
 
Recall that in Subsection~\ref{subsec:par} we explained why Proposition~\ref{p:main}
implies Theorem~\ref{thm:main}. Thus the argument that follows will complete the proof
of our main result. 

\begin{proof}[Proof of Proposition~\ref{p:main}]
	Fix a partition $V=V_1\dcup V_2\dcup V_3\dcup Z$
	such that the sets $V_1$, $V_2$, and $V_3$ are independent and 
	subject to this $|Z|$ is as small as possible. Define the real numbers 
	$\gamma$ and~$z$ by $|E|=\gamma n^2$ and $|Z|=zn$. Owing to the Lemmata~\ref{l:33}
	and~\ref{l:41} we already know 
	\begin{equation}\label{eq:0027}
		\gamma\ge\frac 7{24} 
		\quad \text{ as well as } \quad 
		z\le 1-\frac 1{3(1-2\gamma)}\le \frac 15\,.
	\end{equation}

	We proceed by verifying the main hypothesis of Lemma~\ref{l:42}.
	
	\begin{clm}\label{clm:1}
		If $i, j\in \ZZ/3\ZZ$ are distinct, then $|V_i|+|V_j|\ge \tfrac12 n$.
	\end{clm}
	
	\begin{proof}
		If there are two exceptions, say $|V_1|+|V_2|<\frac12n$ and $|V_1|+|V_3|<\frac12n$,
		then 
		\begin{align*}
			|V_2|+|V_3|
			&=
			2n-2|Z|-(|V_1|+|V_2|)-(|V_1|+|V_3|) \\
			&>
			n-2|Z|
			\overset{\eqref{eq:0027}}{\ge}
			\frac{2n}{3(1-2\gamma)}-n
			\overset{\eqref{eq:0027}}{\ge}
			2\gamma n
		\end{align*}
		contradicts Lemma~\ref{l:37}. 
		
		So by symmetry it suffices to refute that $|V_2|+|V_3|<\frac12n$, while
		$|V_1|+|V_2|, |V_1|+|V_3|\ge\frac12n$.
		Now every set $P\subseteq V_2$ of size $|P|=\frac12n-|V_1|$ satisfies 
		$e(P, V_1)\ge \frac1{18}n^2$ and by averaging we infer 
		\[
			e(V_1, V_2)\ge \frac{|V_2|n^2}{9(n-2|V_1|)}\,.
		\]
		A similar estimate holds with $V_3$ instead of $V_2$. Adding them both 
		and combining the result with Lemma~\ref{p:31} we get
		\[
			\frac{|V_2|+|V_3|}{n-2|V_1|}\le 18\gamma-4\,.
		\]

		Intending to solve this for $t=|V_2|+|V_3|$ we rewrite the denominator as
		$2t-(1-2z)n$ and obtain $(1-2z)(18\gamma-4)n\le 9(4\gamma-1)t$. 
		In view of $t<\frac12n$ and~\eqref{eq:0027} this implies 
		\[	
			\frac{(6\gamma-1)(18\gamma-4)}{3(1-2\gamma)}<\frac{9(4\gamma-1)}2\,,
		\]
		which rewrites as $(24\gamma-7)(18\gamma-5)<0$.
		This contradiction to $\gamma\ge\frac7{24}$ concludes the proof of Claim~\ref{clm:1}.
	\end{proof}
	
	Now Lemma~\ref{l:42} tells us that
	\begin{equation}\label{eq:ende}
		\gamma
		\ge
		\frac{1+z+z^2}{3(1+2z)}\,.
	\end{equation}

	Our next goal is to derive an upper bound on $\gamma$ in terms of $z$. 
	To this end we shall estimate the degrees of the vertices in $V_i$ as follows. 
	
	\begin{clm}\label{clm:0044}
		If $i\in\ZZ/3\ZZ$ and $x_i\in V_i$, then 
		\[
			d(x_i)
			\le 
			\frac n3+\frac{5(|V_{i+1}|+|V_{i-1}|)}8-\frac{n^2}{18(|V_{i+1}|+|V_{i-1}|)}\,.
		\]
	\end{clm}

%

	\begin{proof}
		In terms of the function $h(t)=\frac13 n+\frac58 t-\frac{n^2}{18t}$, which 
		is defined for all $t>0$, we are to prove $d(x_i)\le h(|V_{i+1}|+|V_{i-1}|)$.
		Since $t$ is increasing, Claim~\ref{clm:1} tells us 
		\[
			h(|V_{i+1}|+|V_{i-1}|)\ge h(\tfrac12 n)>\tfrac12 n
		\]
		and thus we may assume that $d(x_i)\ge \frac12 n$. Due to Lemma~\ref{lem:1243}
		there are two disjoint independent sets $V'_{i-1}, V'_{i+1}\subseteq N(x_i)$ such that 
		$t'=|V'_{i-1}|+|V'_{i+1}|$ satisfies $d(x_i)\le h(t')$. As $V_i$ is independent, the 
		sets $V_i$, $V'_{i-1}$, and $V'_{i+1}$ are mutually disjoint and thus the minimality 
		of $|Z|$ implies $t'\le |V_{i-1}|+|V_{i+1}|$. So the monotonicity of $h$ yields 
		indeed $d(x_i)\le h(|V_{i-1}|+|V_{i+1}|)$.
	\end{proof}
	
	Taking into account that due to Lemma~\ref{l:37} the vertices in $Z$ have degree 
	at most $(\frac32\gamma+\frac16)n$ we now obtain
	\begin{align*}
		2\gamma n^2
		&=
		\sum_{x\in V}d(x) \\
		&\le
		\sum_{i\in \ZZ/3\ZZ}|V_i|\left(\frac n3+\frac{5(|V_{i+1}|+|V_{i-1}|)}8-\frac{n^2}{18(|V_{i+1}|+|V_{i-1}|)}\right)+|Z|(\tfrac32\gamma+\tfrac16)n\,.
	\end{align*}
	In view of $|V_1||V_2|+|V_2||V_3|+|V_3||V_1|\le \frac13(n-|Z|)^2$ and Nesbitt's
	inequality 
	\[
		\frac{|V_1|}{|V_2|+|V_3|}+\frac{|V_2|}{|V_3|+|V_1|}+\frac{|V_3|}{|V_1|+|V_2|}
		\ge
		\frac32
	\]
	this leads us to 
	\[
		(2-\tfrac32z)\gamma
		\le 
		\tfrac13(1-z)+\tfrac{5}{12}(1-z)^2-\tfrac1{12}+\tfrac16 z
		=
		(2-\tfrac32z)(\tfrac13-\tfrac z4)+\tfrac1{24}z^2\,,
	\]
	whence 
	\begin{equation}\label{eq:0039}
		\gamma
		\le 
		\frac13-\frac z4+\frac{z^2}{48-36z}
		\overset{\eqref{eq:0027}}{\le}
		\frac13-\frac z4+\frac{z^2}{36}\,.
	\end{equation}

	Together with~\eqref{eq:ende} this demonstrates
	\[
		\frac{1+z+z^2}{3(1+2z)}
		\le
		\frac13-\frac z4+\frac{z^2}{48-36z}
	\]
	and thus
	\[
		0\le z(2-21z+16z^2)\,,
	\]
	which is easily verified to imply 
	\begin{equation}\label{eq:45}
		z\le\tfrac 3{29}
	\end{equation}
	(or $z>1$, but this would contradict~\eqref{eq:0027}). Next we plan to split $Z$
	into three parts and to adjoin these to $V_1$, $V_2$, and $V_3$, thus creating the 
	situation considered in Lemma~\ref{l:43}. Of course it is recommendable to move every 
	vertex $z\in Z$ into a vertex class $V_i$ it has only few neighbours in and the
	subsequent claim will help us with the analysis of this process. 
	
	\begin{clm}
		If $x\in Z$, then $\min\bigl\{|N(x)\cap V_i|\colon i\in \ZZ/3\ZZ\bigr\}\le \frac 2{17}n$.
	\end{clm}
	
	\begin{proof}
		Assume contrariwise that each of the three sets $K_i=N(x)\cap V_i$ satisfies
		$|K_i| > \frac 2{17}n$. By Corollary~\ref{c:26} we have $|V_{i+1}|\le \frac 13n$ for 
		every $i\in \ZZ/3\ZZ$ and, hence, there exists a set $L_i\subseteq K_i$
		such that $ \frac 2{17}n < |L_i|\le \frac 12n- |V_{i+1}|$. Moreover, 
		Claim~\ref{clm:1} shows that there is a set~$M_i$ such that $L_i\subseteq M_i\subseteq V_i$
		and $|M_i|=\frac12n- |V_{i+1}|$. Now $e(M_i\dcup V_{i+1})\ge \frac 1{18}n^2$ implies 
		that the number of missing edges between $M_i$ and $V_{i+1}$ is at 
		most $(\frac12n- |V_{i+1}|)|V_{i+1}|-\frac 1{18}n^2$. In particular, there are
		at most that many missing edges between $L_i$ and $L_{i+1}$, for which reason		%
		\[
			 \frac{|L_i||L_{i+1}|-e(L_i, L_{i+1})}{|L_i||L_{i+1}|}
			 <
			 \frac{289}{4n^2}
			 \left(\left(\frac n2- |V_{i+1}|\right)|V_{i+1}|-\frac {n^2}{18}\right)\,.
		\]
		Summing all three such estimates and 
		exploiting $|V_1|^2+|V_2|^2+|V_3|^2\ge \frac13 (n-|Z|)^2$ we obtain
		\[
			\sum_{i\in\ZZ/3\ZZ}\frac{|L_i||L_{i+1}|-e(L_i, L_{i+1})}{|L_i||L_{i+1}|}
			<
			\frac{289}{4}\left(\frac{1-z}{2}-\frac{(1-z)^2}{3}-\frac 16\right)
			=
			\frac{289z(1-2z)}{24}
			\overset{\eqref{eq:45}}{<}
			1
		\]
		and thus
		\begin{equation}\label{eq:46}
			\sum_{i\in\ZZ/3\ZZ}|L_{i-1}|\bigl(|L_i||L_{i+1}|-e(L_i, L_{i+1})\bigr)
			<
			|L_1||L_2||L_3|\,.
		\end{equation}

		Let us now look at all $|L_1||L_2||L_3|$ 
		triples $(\ell_1, \ell_2, \ell_3)\in L_1\times L_2\times L_3$.
		For every $i\in\ZZ/3\ZZ$ there are exactly 
		$|L_{i-1}|\bigl(|L_i||L_{i+1}|-e(L_i, L_{i+1})\bigr)$ such triples 
		with $\ell_i\ell_{i+1}\not\in E$. So by~\eqref{eq:46} at least one such triple 
		yields a triangle in $G$. As the vertex $x$ extends this triangle to a $K_4$, 
		we have thereby reached a contradiction. 
	\end{proof}
	
	Next we form a partition $V=A_1\dcup A_2\dcup A_3$ such that 
	\begin{enumerate}
		\item[$\bullet$] $A_i\supseteq V_i$ for all $i\in \ZZ/3\ZZ$
		\item[$\bullet$] and $|N(x_i)\cap V_i|\le \frac 2{17}n$ for all $x\in A_i\sm V_i$.
	\end{enumerate}
	The main property of this construction is that the number $\omega$ defined by 
	\[
		\omega n^2=e(A_1)+e(A_2)+e(A_3)
	\]
	is small. In fact a straightforward reasoning shows
	\begin{equation}\label{eq:47}
		\omega
		\le 
		(\tfrac 2{17}n |Z|+e(Z))n^{-2}
		\le 
		\tfrac 2{17}z+\tfrac 13z^2
		\overset{\eqref{eq:45}}{<}
		\tfrac7{46}z
		<\tfrac1{60}\,,
	\end{equation}
	and, therefore, Lemma~\ref{l:43} discloses $36\gamma\ge 12-58\omega$. 
	Together with~\eqref{eq:0039} and~\eqref{eq:47} we conclude
	\[
		12-(9-\tfrac 4{23})z\le 12-58\omega\le 12-9z+z^2\,,
	\]
	i.e., $z(\frac4{23}-z)\le 0$. Now~\eqref{eq:45} reveals $z=0$,~\eqref{eq:ende}
	tells us $\gamma\ge \frac13$, and by Tur\'an's theorem $G$ is indeed a tripartite 
	Tur\'an graph. 
\end{proof}

\subsection*{Acknowledgement} I would like to thank the referee for reading this 
article very carefully and providing valuable remarks.


\begin{bibdiv}
\begin{biblist}
\bib{A}{article}{
   author={Andr\'{a}sfai, B.},
   title={\"{U}ber ein Extremalproblem der Graphentheorie},
   language={German},
   journal={Acta Math. Acad. Sci. Hungar.},
   volume={13},
   date={1962},
   pages={443--455},
   issn={0001-5954},
   review={\MR{145503}},
   doi={10.1007/BF02020809},
}

\bib{BMPP2018}{article}{
   author={Bedenknecht, Wiebke},
   author={Mota, Guilherme Oliveira},
   author={Reiher, Chr.},
   author={Schacht, Mathias},
   title={On the local density problem for graphs of given odd-girth},
   journal={J. Graph Theory},
   volume={90},
   date={2019},
   number={2},
   pages={137--149},
   issn={0364-9024},
   review={\MR{3891931}},
   doi={10.1002/jgt.22372},
}

\bib{BE76}{article}{
   author={Bollob{\'a}s, B{\'e}la},
   author={Erd{\H{o}}s, Paul},
   title={On a Ramsey-Tur\'an type problem},
   journal={J. Combinatorial Theory Ser. B},
   volume={21},
   date={1976},
   number={2},
   pages={166--168},
   review={\MR{0424613}},
}


\bib{B10}{article}{
   author={Brandt, Stephan},
   title={Triangle-free graphs whose independence number equals the degree},
   journal={Discrete Math.},
   volume={310},
   date={2010},
   number={3},
   pages={662--669},
   issn={0012-365X},
   review={\MR{2564822}},
   doi={10.1016/j.disc.2009.05.021},
}

\bib{CG}{article}{
   author={Chung, F. R. K.},
   author={Graham, R. L.},
   title={On graphs not containing prescribed induced subgraphs},
   conference={
      title={A tribute to Paul Erd\H{o}s},
   },
   book={
      publisher={Cambridge Univ. Press, Cambridge},
   },
   date={1990},
   pages={111--120},
   review={\MR{1117008}},
}

\bib{E1}{article}{
   author={Erd\H{o}s, P.},
   title={Problems and results in graph theory and combinatorial analysis},
   conference={
      title={Proceedings of the Fifth British Combinatorial Conference},
      address={Univ. Aberdeen, Aberdeen},
      date={1975},
   },
   book={
      publisher={Utilitas Math., Winnipeg, Man.},
   },
   date={1976},
   pages={169--192. Congressus Numerantium, No. XV},
   review={\MR{0409246}},
}

\bib{Er97a}{article}{
   author={Erd{\H{o}}s, Paul},
   title={Some old and new problems in various branches of combinatorics},
   note={Graphs and combinatorics (Marseille, 1995)},
   journal={Discrete Math.},
   volume={165/166},
   date={1997},
   pages={227--231},
   issn={0012-365X},
   review={\MR{1439273}},
   doi={10.1016/S0012-365X(96)00173-2},
}
		
\bib{EFRS}{article}{
   author={Erd\H{o}s, P.},
   author={Faudree, R. J.},
   author={Rousseau, C. C.},
   author={Schelp, R. H.},
   title={A local density condition for triangles},
   note={Graph theory and applications (Hakone, 1990)},
   journal={Discrete Math.},
   volume={127},
   date={1994},
   number={1-3},
   pages={153--161},
   issn={0012-365X},
   review={\MR{1273598}},
   doi={10.1016/0012-365X(92)00474-6},
}

\bib{EHSS}{article}{
   author={Erd{\H{o}}s, P.},
   author={Hajnal, A.},
   author={S{\'o}s, Vera T.},
   author={Szemer{\'e}di, E.},
   title={More results on Ramsey-Tur\'an type problems},
   journal={Combinatorica},
   volume={3},
   date={1983},
   number={1},
   pages={69--81},
   issn={0209-9683},
   review={\MR{716422}},
   doi={10.1007/BF02579342},
}

\bib{FLZ15}{article}{
   author={Fox, Jacob},
   author={Loh, Po-Shen},
   author={Zhao, Yufei},
   title={The critical window for the classical Ramsey-Tur\'an problem},
   journal={Combinatorica},
   volume={35},
   date={2015},
   number={4},
   pages={435--476},
   issn={0209-9683},
   review={\MR{3386053}},
   doi={10.1007/s00493-014-3025-3},
}

\bib{KeSu06}{article}{
   author={Keevash, Peter},
   author={Sudakov, Benny},
   title={Sparse halves in triangle-free graphs},
   journal={J. Combin. Theory Ser. B},
   volume={96},
   date={2006},
   number={4},
   pages={614--620},
   issn={0095-8956},
   review={\MR{2232396}},
   doi={10.1016/j.jctb.2005.11.003},
}


\bib{K}{article}{
   author={Krivelevich, Michael},
   title={On the edge distribution in triangle-free graphs},
   journal={J. Combin. Theory Ser. B},
   volume={63},
   date={1995},
   number={2},
   pages={245--260},
   issn={0095-8956},
   review={\MR{1320169}},
   doi={10.1006/jctb.1995.1018},
}

\bib{LM}{article}{
   author={Liu, Xizhi},
   author={Ma, Jie},
   title={Sparse halves in $K_4$-free graphs},
   journal={J. Graph Theory},
   volume={99},
   date={2022},
   number={1},
   pages={5--25},
   issn={0364-9024},
   review={\MR{4371476}},
   doi={10.1002/jgt.22722},
}

\bib{LPR2}{article}{
   author={\L uczak, Tomasz},
   author={Polcyn, Joanna},
   author={Reiher, Chr.},
   title={On the Ramsey-Tur\'{a}n density of triangles},
   journal={Combinatorica},
   volume={42},
   date={2022},
   number={1},
   pages={115--136},
   issn={0209-9683},
   review={\MR{4426766}},
   doi={10.1007/s00493-021-4340-0},
}

\bib{LR-a}{article}{
	author={L\"{u}ders, Clara Marie},
	author={Reiher, Chr.},
	title={The Ramsey--Tur\'{a}n problem for cliques},
	journal={Israel J. Math.},
	volume={230},
	date={2019},
	number={2},
	pages={613--652},
	issn={0021-2172},
	review={\MR{3940430}},
	doi={10.1007/s11856-019-1831-4},
}		

\bib{M}{article}{
	author={Mantel, W.},
	title={Problem 28 (Solution by H. Gouwentak, W. Mantel, J. Teixeira de Mattes, 
		F. Schuh and W. A. Wythoff)},
	journal={Wiskundige Opgaven},
	date={1907},
	number={10},
	pages={60--61}
}

\bib{NY}{article}{
   author={Norin, Sergey},
   author={Yepremyan, Liana},
   title={Sparse halves in dense triangle-free graphs},
   journal={J. Combin. Theory Ser. B},
   volume={115},
   date={2015},
   pages={1--25},
   issn={0095-8956},
   review={\MR{3383248}},
   doi={10.1016/j.jctb.2015.04.006},
}
	
\bib{R}{article}{
   author={Razborov, A. A.},
   title={More about sparse halves in triangle-free graphs},
   language={Russian, with Russian summary},
   journal={Mat. Sb.},
   volume={213},
   date={2022},
   number={1},
   pages={119--140},
   issn={0368-8666},
   review={\MR{4360109}},
   doi={10.4213/sm9615},
}

\bib{Sz72}{article}{
   author={Szemer{\'e}di, Endre},
   title={On graphs containing no complete subgraph with $4$ vertices},
   language={Hungarian},
   journal={Mat. Lapok},
   volume={23},
   date={1972},
   pages={113--116 (1973)},
   issn={0025-519X},
   review={\MR{0351897}},
}


\bib{T}{article}{
	author={Tur\'an, Paul},
	title={On an extremal problem in graph theory},
	journal={Matematikai \'es Fizikai Lapok (in Hungarian)},
	date={1948},
	pages={436--452}
}
\end{biblist}
\end{bibdiv}
\end{document}